\documentclass{amsart}%
\usepackage{amsfonts}
\usepackage{amsmath}
\usepackage{amssymb}
\usepackage{graphicx}
\usepackage{version}
\usepackage{caption}
\usepackage{setspace}
\theoremstyle{plain}
\newtheorem{theorem}{Theorem}[section]

\newtheorem{proposition}{Proposition}[section]
\theoremstyle{definition}

\newtheorem{remark}{Remark}[section]
\newtheorem*{notation*}{Notation}
\numberwithin{equation}{section}

\excludeversion{comment1}
\DeclareMathOperator*{\esssup}{ess\,sup}
\begin{document}
\title{A Fractal Operator on Some Standard Spaces of Functions}
\author[P. Viswanathan]{P. Viswanathan}
\address{Australian National University}
\author[M.A. Navascu\'{e}s]{M.A. Navascu\'{e}s}
\address{Universidad de Zaragoza}

\begin{abstract}
By appropriate  choices of elements in the underlying iterated function system, methodology of fractal interpolation entitles one to associate a family of continuous  self-referential functions  with  a prescribed  real-valued continuous function  on a  real compact interval. This procedure elicits
 what is referred to as $\alpha$-fractal operator on $\mathcal{C}(I)$, the space of all real-valued continuous functions defined on a compact interval $I$. With an eye towards connecting fractal functions with other branches of mathematics, in this article, we continue to investigate fractal operator in more general  spaces such as the space $\mathcal{B}(I)$ of all bounded functions  and Lebesgue  space $\mathcal{L}^p(I)$, and some  standard  spaces of smooth functions such as the space $\mathcal{C}^k(I)$ of $k$-times continuously differentiable functions,  H\"{o}lder spaces $\mathcal{C}^{k,\sigma}(I)$, and Sobolev spaces $\mathcal{W}^{k,p}(I)$. Using properties of the $\alpha$-fractal operator, the existence of Schauder basis consisting of self-referential functions  is established.
\end{abstract}

\maketitle



\section{Introduction}

It is well-known that Fractal Interpolation Function (FIF), a notion introduced by Barnsley \cite{Barnsley1,BH},  provides an effective tool for the approximation of rough functions and an alternative to traditional nonrecursive  approximation methods that generally uses smooth functions. Fractal function was originally introduced as a continuous function interpolating a prescribed set of data points in the Cartesian plane. These fractal approximants are constructed as attractors of appropriate Iterated Function Systems (IFSs). Furthermore, a fractal function is the fixed point of a suitable contraction map (Read-Bajraktarevi\'{c} operator) defined on a subset of the space of all continuous real-valued  functions; see Section \ref{FFS2}.
\par
 Besides being a source of smooth and non-smooth approximants catering to the modeling problem at hand, the notion of FIF provides a bounded linear operator, termed $\alpha$-fractal operator, on the space  $\mathcal{C}(I)$ of all continuous real-valued  functions on a compact interval $I \subset \mathbb{R}$. In earlier papers \cite{N1,N3,N9,N10a}, the second author  studied this
 fractal operator in detail.  This operator theoretic formalism of fractal functions enables them to interact with other traditional branches of pure and applied mathematics including functional analysis, operator theory, complex analysis, harmonic analysis, and approximation theory.
\par
In  the current article, we continue to explore  the aforementioned fractal operator, however, within the setting of function spaces such as the space of bounded functions $\mathcal{B}(I)$, Lebesgue spaces $\mathcal{L}^p(I)$, space of $k$-times continuously differentiable functions $\mathcal{C}^k(I)$,  H\"{o}lder spaces $\mathcal{C}^{k,\sigma}(I)$, and Sobolev spaces $\mathcal{W}^{k,p}(I)$ that are prevalent in modern analysis and function theory. Hence, it is our view that the present work enriches the theory of fractal functions and facilitate them to find further applications in various fields such as numerical analysis, functional analysis, harmonic analysis, for instance,  in connection with PDEs. In particular, we expect that the current study will pave the way towards the investigation of shape preserving fractal approximation in various function  spaces considered herein. For shape preserving fractal approximation in the space of continuous functions, the reader is invited to refer \cite{VCN}.
\par
The research reported here is admittedly influenced to an extent by works on fractal functions and local fractal functions by Massopust; see \cite{PM1,PM2}. However, it should be noted that our exposition aims at a different goal. Our main observations are centered on the $\alpha$-fractal operator defined on standard spaces of functions, in contrast to the Read-Bajraktarevi\'{c} operator that define fractal functions, and we find this link quite intriguing.
\par
Turning to the layout of our paper, in Section \ref{FFS2}, we assemble the requisite general material. In Section \ref{FFS3}, we construct $\alpha$-fractal functions in various spaces of functions. Section \ref{FFS4} concerns development of some foundational aspects of the associated fractal operator.
\section{Notation and Preliminaries} \label{FFS2}
In this section we  provide a rudimentary introduction to function spaces dealt in this article. Further, we
review the requisite background material on fractal functions.  We recommend the reader to consult the references  \cite{Barnsley1,N1,T1,T2} for further details.
\par
 The set of natural numbers will be denoted by $\mathbb{N}$, whilst the set of real numbers by $\mathbb{R}$.  For a fixed $N \in \mathbb{N}$, we shall write $\mathbb{N}_N$ for the set of first $N$ natural numbers. By a self-map we mean a function whose domain and codomain are same. For a given compact interval $I=[c,d] \subset \mathbb{R}$, let $$\mathcal{B}(I):= \big\{g: I\to \mathbb{R}: ~g~ \text{is bounded on}~ I\big\}.$$
 The functional
 $$\|g\|_\infty := \sup\big\{|g(x)|: x \in I\big\},$$
 termed uniform norm, turns  $\mathcal{B}(I)$ to a Banach space.
For $k \in \mathbb{N} \cup \{0\}$, consider  the space
$$\mathcal{C}^k(I):= \big\{g: I\to \mathbb{R}: ~g~ \text{is}~ k \text{-times differentiable  and}~ g^{(k)} \in \mathcal{C}(I)\big\}$$
endowed with the norm
$$\|g||_{\mathcal{C}^k} := \max \big\{\|g^{(r)}\|_\infty: r \in \mathbb{N}_k \cup\{0\}\big\}.$$
In addition to the Banach spaces $\mathcal{B}(I)$ and $\mathcal{C}^k(I)$, we require the following spaces that have numerous applications in analysis. \\
For $0 < p \le \infty$, let $$ \mathcal{L}^p(I):=  \big\{g: I \to \mathbb{R}: ~g~ \text{is measurable and}~\|g\|_{p} < \infty \big\},$$
 where the ``norm" is given by
\begin{equation*}
 \|g\|_ p= \left\{\begin{array}{rll}
& \big[\int \limits_{I}|g(x)|^p ~\mathrm{d}x\big]^{1/p} , \; 0 < p < \infty,\\
&  \esssup_{x\in I} |g(x)|, \; \; \; \; p =\infty.
\end{array}\right.
\end{equation*}
We recall that for $1 \le p \le \infty$,  $\|.\|_ p$ defines a norm on $\mathcal{L}^p(I)$, and $\big(\mathcal{L}^p(I),\|.\|_p\big)$  is a Banach space. Note that for $p=2$, the space $\mathcal{L}^2(I)$ is a Hilbert space with respect to the inner product
 $$ <g_1,g_2> := \int_{I} g_1 g_2~ \mathrm{d}x$$
 For $ 0 <p <1$, $\|.\|_p$ is not really a norm, but only a quasi-norm, that is, in the place of triangle inequality we have
$$\|g_1+g_2\|_p \le 2^{\frac{1}{p}} \big(\|g_1\|_p + \|g_2\|_p \big)$$
and $\mathcal{L}^p(I)$ is a quasi-Banach space.\\
Let $u,v \in \mathcal{L}^1(I)$. We say that $v$ is $k$-th weak derivative of $u$ and write $D^k u=v$, provided
$$ \int_{I} u~ D^k \varphi~ \mathrm{d}x = (-1)^k \int_{I} v\varphi~ \mathrm{d}x$$
for all infinitely differentiable functions $\varphi$ with $\varphi(c)=\varphi(d)=0$. \\
For $1 \le p \le \infty$ and $k \in \mathbb{N} \cup \{0\}$, let $\mathcal{W}^{k,p}(I)$ denote the usual Sobolev space. That is,
$$\mathcal{W}^{k,p}(I) := \big\{g: I \to \mathbb{R}:  D^jg \in \mathcal{L}^{p} (I), ~ j=0,1,\dots, k\},$$
 where $D^jg$ denotes the $j$-th weak or distributional derivative of $g$.  The space $\mathcal{W}^{k,p}(I)$  endowed with the norm
 \begin{equation*}
\|g\|_{\mathcal{W}^{k,p}}:= \left\{\begin{array}{rllll}
&\Big[\sum_{j=0}^k \| D^j g \|_p^p \Big]^{\frac{1}{p}}, \; \text{for} ~~ p \in [1, \infty),\\
& \sum_{j=0}^k \|D^j g\|_\infty, \; \; \; \;\; \;\text{for} ~~p =\infty.\\
\end{array}\right.
\end{equation*}
 is a Banach space. For $p=2$, it is customary to denote  $\mathcal{W}^{k,p}(I)$ by $\mathcal{H}^k(I)$, which is a Hilbert space. \\
A function $u: I \to \mathbb{R}$ is said to be  H\"{o}lder continuous with exponent $\sigma$ if
$$|u(x)-u(y)| \le C |x-y|^\sigma,\quad \text {for all} \quad x,y \in I \quad \text{and for some}\quad  C>0.$$
 For H\"{o}lder continuous functions $u$ with exponent $\sigma$, let us define $\sigma$-th  H\"{o}lder semi-norm  as $$[u]_\sigma = \sup_{ x,y \in I,~ x \neq y} \frac{|u(x)-u(y)|}{|x-y|^\sigma}$$
and consider the H\"{o}lder space
$$\mathcal{C}^{k,\sigma} (I) :=\big\{ g \in \mathcal{C}^{k-1}(I): g^{(k)}~~ \text{is H\"{o}lder continuous with exponent}~~ \sigma \big\}.$$
The space $\mathcal{C}^{k,\sigma}(I)$ is a Banach space when endowed with the norm
$$\|g\|_{\mathcal{C}^{k,\sigma}}:= \sum_{j=0}^k\|g^{(j)}\|_\infty + [g^{(k)}]_\sigma.$$
\par
Having exposed the reader with function spaces that we shall encounter in due course, next we  provide an overview  of  fractal interpolation and related ideas.  Assume that  $N \in  \mathbb{N}$, $N>2$. Let  $\big\{(x_i, y_i) \in \mathbb{R}^2: i \in \mathbb{N}_N\big\}$ denote the prescribed set of interpolation data with strictly increasing abscissae. Set $I = [x_1,x_N]$ and $I_i = [x_i,x_{i+1}]$ for $i \in \mathbb{N}_{N-1}$.  Fractal interpolation constructs   a continuous function $g: I \to \mathbb{R}$ satisfying $g(x_i)=y_i$ for all $i \in \mathbb{N}_N$ and whose graph $G(g)$ is a fractal in the sense that $G(g)$ is a union of transformed copies of itself (see (\ref{SAINsr})). Suppose that $L_i: I \rightarrow I_i$, $i \in \mathbb{N}_{N-1}$ are affinities  satisfying
\begin{equation*}\label{ICAFWeq1}
L_i(x_1)=x_i,\quad L_i(x_N)=x_{i+1}.
\end{equation*}
For $i\in \mathbb{N}_{N-1}$, let
 $F_i: I \times \mathbb{R} \to \mathbb{R}$ be continuous functions fulfilling
 \begin{equation*}\label{ICAFWeq2}
\big|F_i(x,y)-F_i(x,y^*)\big| \le c_i |y-y^*|, \quad F_i(x_1,y_1)=y_i,~~~ F_i(x_N,y_N)=y_{i+1}.
\end{equation*}
for $y,y^* \in \mathbb{R}$ and $0<c_i<1$. Define
\begin{equation*}
\mathcal{C}_{y_1,y_N}(I)= \{h \in \mathcal{C}(I): h(x_1)=y_1, h(x_N)=y_N\}.
\end{equation*}
It is readily observed that $\mathcal{C}_{y_1,y_N}(I)$ is a closed (metric) subspace of the Banach space $\big(\mathcal{C}(I), \|.\|_\infty \big)$. Define the Read-Bajraktarevi\'{c} (RB) operator
$T: \mathcal{C}_{y_1,y_N}(I) \to  \mathcal{C}_{y_1,y_N}(I)$ via
$$ (Th) (x) = F_i \big(L_i^{-1}(x), h \circ L_i^{-1}(x)\big), ~ x \in I_i, ~ i \in \mathbb{N}_{N-1}.$$
The nonlinear mapping $T$ is a contraction with a contraction factor $c:= \max\{c_i: i \in \mathbb{N}_{N-1}\}$. Consequently, by the Banach fixed point theorem, $T$ has a unique fixed point, say  $g$. Further, $g$ interpolates the data $\{(x_i,y_i): i \in \mathbb{N}_N\}$ and enjoys the functional equation
$$g\big(L_i(x)\big)= F_i\big(x, g(x)\big), \quad  x\in I, \quad i \in \mathbb{N}_{N-1}.$$
 Let $w_i: I\times \mathbb{R} \to I_i \times \mathbb{R} \subset I\times \mathbb{R}$ be defined by $$w_i(x,y)=\big(L_i(x),F_i(x,y)\big), \quad i \in \mathbb{N}_{N-1}.$$  Consider the iterated function system (IFS) $W=\big\{I \times \mathbb{R}; w_i, i \in \mathbb{N}_{N-1}\big\}$. Then the \emph{attractor} $\mathrm{G}$ of the IFS $W$  is the graph $G(g)$ of the function $g$, that is,
\begin{equation}\label{SAINsr}
G(g)= \cup_{i \in \mathbb{N}_{N-1}} w_i \big (G(g) \big).
\end{equation}
Therefore, $g$ is a \emph{self-referential function}.
\par
Navascu\'{e}s \cite{N1} observed that the theory of FIF can be used to generate a family of continuous functions having fractal characteristics from a prescribed continuous function.   To this end, for a given $f \in \mathcal{C}(I)$, consider a partition $\Delta:=\{x_1,x_2,\dots, x_N\}$ of $I$ satisfying $x_1<x_2<\dots<x_N$, a continuous  map $ b: I \to \mathbb{R}$ such that $b \neq f$, $b(x_1)=f(x_1)$ and $b(x_N)=f(x_N)$, and  $N-1$ real numbers $\alpha_i$ satisfying $|\alpha_i|<1$. Define an IFS through the maps
\begin{equation}\label{ICAFWeq4}
L_i(x)=a_ix+d_i, \quad F_i(x,y)=\alpha_i y + f \circ L_i(x) -\alpha_i b(x), \quad
i\in \mathbb{N}_{N-1}.
\end{equation}
 The corresponding FIF denoted by
$f^{\alpha}_{\Delta, b}= f^\alpha$ is referred to as $\alpha$-fractal function for $f$ (fractal perturbation of $f$) with respect to a scale vector $\alpha=(\alpha_1,\alpha_2,\dots, \alpha_{N-1})$, base function $b$, and  partition $\Delta$. Here the interpolation points  are $ \big\{(x_i,f(x_i)): i \in \mathbb{N}_N\big\}$. The fractal dimension (Minkowski dimension or Hausdorff dimension) of the function $f^\alpha$  depends on the parameter $\alpha \in (-1,1)^{N-1}$. The function $f^\alpha$  is the fixed point of the operator $T_{\Delta,b,f}^\alpha: \mathcal{C}_f (I) \to \mathcal{C}_f (I)$ defined by
\begin{equation}\label{RBCS}
(T_{\Delta,b,f}^\alpha g) (x) = f(x) + \alpha_i \big(g-b) \circ L_i^{-1}(x), \quad x \in I_i, \quad i \in \mathbb{N}_{N-1},
\end{equation}
where $\mathcal{C}_f (I) := \big\{g \in \mathcal{C}(I): ~g(x_1)=f(x_1), ~g(x_N)=f(x_N)\big\}$. Consequently, the $\alpha$-fractal function corresponding to $f$ satisfies the self-referential equation
\begin{equation}\label{ICAFWeq4}
f_{\Delta,b}^\alpha(x)= f(x) + \alpha_i \big(f_{\Delta,b}^\alpha-b\big)(L_i^{-1}(x)), \quad x \in I_i, \quad i \in \mathbb{N}_{N-1}.
\end{equation}
 Assume that the base function $b$ depends linearly on $f$, say $b=Lf$, where $L: \mathcal{C}(I) \to \mathcal{C}(I)$  is a bounded linear map.  Then the following map  referred to as \emph{$\alpha$-fractal operator}
$$\mathcal{F}_{\Delta,b}^\alpha: \mathcal{C}(I) \rightarrow \mathcal{C}(I), \quad \mathcal{F}_{\Delta,b}^\alpha (f)=f_{\Delta,b}^\alpha$$ is a bounded linear operator.\\
To obtain fractal functions with more flexibility, iterated function system wherein scaling factors are replaced by scaling functions received attention in the recent literature on fractal functions. That is, one may consider the IFS with maps
\begin{equation}\label{Sain1}
L_i(x)=a_ix+d_i,\quad  F_i(x,y)=\alpha_i(x) y + f \circ L_i(x) -\alpha_i(x) b(x),\quad
i\in \mathbb{N}_{N-1},
\end{equation}
 where $\alpha_i$ are continuous functions satisfying $\max\{\|\alpha_i\|_\infty: i \in \mathbb{N}_{N-1}\} <1$,   and $b \neq f$ is a continuous function that agrees with $f$ at the extremes of the interval $I$ (see, for instance, \cite{WY}).
The corresponding  $\alpha$-fractal function is the fixed point of the RB-operator
\begin{equation}\label{RB}
(T_{\Delta,b,f}^\alpha g)(x) = f(x) + \alpha_i \big(L_i^{-1}(x)\big) (g-b)\big(L_i^{-1}(x) \big),\quad x \in I_i,\quad i \in \mathbb{N}_{N-1},
\end{equation}
and hence obeys the functional equation
\begin{equation}\label{Sain2}
f_{\Delta,b}^\alpha(x)= f(x) + \alpha_i\big(L_i^{-1}(x)) \big(f_{\Delta,b}^\alpha-b) \big(L_i^{-1}(x)\big), \quad x \in I_i,\quad  i \in \mathbb{N}_{N-1}.
\end{equation}
\section{$\alpha$-Fractal Functions on Various Function Spaces}\label{FFS3}
In this section, by considering suitable conditions on the scaling functions $\alpha_i$ and base function $b$,
we construct $\alpha$-fractal function $f_{\Delta,b}^\alpha$ corresponding to $f$ in various function spaces. The idea is to choose the parameters so that the map $T_{\Delta,b,f}^\alpha$ becomes a contraction in the space under consideration and $f_{\Delta,b}^\alpha$ is the corresponding fixed point, whose existence is ensured by the Banach fixed point theorem. For the sake of notational simplicity, we may suppress the dependence on parameters to denote $T_{\Delta,b,f}^\alpha$ by $T$, $f_{\Delta,b}^\alpha$ by $f^\alpha$ and the corresponding fractal operator $\mathcal{F}_{\Delta,b}^\alpha$ by $\mathcal{F}^\alpha$. The
following theorems provide sufficient  conditions on the scaling functions
and  base functions for the  associate fractal function $f^\alpha$ to share the properties of the
original function $f$.
\begin{theorem}\label{FFSB}
Let $ f \in \mathcal{B}(I)$. Assume that $\Delta:= \{x_1,x_2, \dots, x_N\}$ is a partition of~ $I$  with strictly increasing abscissae, and $I_i=[x_i,x_{i+1})$ for $i \in \mathbb{N}_{N-2}$, $I_{N-1}=[x_{N-1},x_N]$ be the corresponding subintervals. Further assume that
the maps $L_i: I \to I_i$ are affinities satisfying $L_i(x_1)=x_i$, $L_i(x_N^-)=x_{i+1}$ for $i \in \mathbb{N}_{N-1}$, and the base function $b$ and scaling functions $\alpha_i$, $i \in \mathbb{N}_{N-1}$ are real-valued bounded functions on $I$. Then the RB-operator defined in (\ref{RB}) is a well-defined self-map on $\mathcal{B}(I)$. Furthermore, if $\|\alpha\|_\infty = \max \{\|\alpha_i\|_\infty: i \in \mathbb{N}_{N-1} \} <1$, then $T$ is a contraction and has a unique fixed point $f^\alpha \in \mathcal{B}(I)$.
\end{theorem}
\begin{proof}
Under the stated hypotheses on functions $b$, $\alpha_i$ and affinities $L_i$, it is straightforward to show that $T$ is well-defined and maps into $\mathcal{B}(I)$. Let $g_1$ and $g_2$ be in $\mathcal{B}(I)$. For $x \in I_i$,
\begin{equation*}
\begin{split}
\big|(Tg_1)(x)-(Tg_2)(x)\big| \le  &~\big |\alpha_i\big(L_i^{-1}(x) \big)\big|~\big|(g_1-g_2)\big(L_i^{-1}(x)\big)\big|\\
\le &~ \|\alpha_i\|_\infty \|g_1-g_2\|_\infty.
\end{split}
\end{equation*}
Therefore
$$\|Tg_1-Tg_2\|_\infty \le \|\alpha\|_\infty \|g_1-g_2\|_\infty,$$
and the assertion follows.
\end{proof}
In the next theorem, we construct $\alpha$-fractal functions in Lebesgue space $\mathcal{L}^p(I)$, for $0 < p \le \infty$. In reference \cite{N9},
 $\alpha$-fractal function in $\mathcal{L}^p(I)$, $1 \le p < \infty$ is obtained as the limit of a sequence of continuous $\alpha$-fractal functions using standard density argument. The present formalism makes the self-referentiality of $\alpha$-fractal functions in $\mathcal{L}^p(I)$ to be explicit and includes   the case of quasi-Banach spaces $\mathcal{L}^p(I)$, $0 <p<1$ and a more general space $\mathcal{L}^\infty(I)$, the space  of all essentially bounded functions. Further, aiming at more generality, we work  with  scaling functions in contrast to the traditional setting of scaling factors.
\begin{theorem} \label{RRAFIFthm1}
Let $f \in \mathcal{L}^p(I)$, $0<p\le \infty$.
Suppose that $\Delta=\{x_1,x_2,\dots, x_N\}$ is a partition of~ $I$ satisfying $x_1<x_2<\dots<x_N$, $I_i:=[x_i,x_{i+1})$ for $i \in \mathbb{N}_{N-2}$ and $I_{N-1}:=[x_{N-1}, x_N]$.
Let $L_i:I \rightarrow I_i$ be  affine maps $L_i(x)=a_i x+d_i$ satisfying $L_i(x_1)=x_i$ and $L_i(x_N^-)=x_{i+1}$ for $i\in \mathbb{N}_{N-1}$. Choose $\alpha_i \in \mathcal{L}^\infty(I)$ for all $i \in \mathbb{N}_{N-1}$ and $b \in \mathcal{L}^p(I)$. Then the RB operator $T$ given in (\ref{RB}) defines a self-map on $\mathcal{L}^p(I)$.
Furthermore, for the scaling functions $\alpha_i$, $i \in \mathbb{N}_{N-1}$ satisfying the following condition,  $T$ is a contraction map on $\mathcal{L}^p(I)$.
\begin{equation*}
\left\{\begin{array}{rllll}
& \Big[\sum \limits_{i \in \mathbb{N}_{N-1}} a_i \|\alpha_i\|_\infty^p\Big]^{1/p}<1 , \; \text{for} ~~ p \in [1, \infty),\\
& \underset{i \in \mathbb{N}_{N-1}}\max~\|\alpha_i\|_\infty <1 , \; \; \; \;\; \;\text{for} ~~p =\infty,\\
&\sum \limits_{i \in \mathbb{N}_{N-1}} a_i \|\alpha_i\|_\infty^p<1 , \; \; \; \;\; \;\text{for} ~~p \in (0,1).
\end{array}\right.
\end{equation*}
Consequently, the corresponding  fixed point $f^\alpha \in \mathcal{L}^p(I)$ obeys the self-referential equation (\ref{Sain2}).
\end{theorem}
\begin{proof}
It follows at once from the hypotheses on the function $f$, affinities $L_i$, scaling functions $\alpha_i$, and base function $b$ that
the  operator $T$ is a well-defined self-map on $\mathcal{L}^p(I)$. It remains to prove that under the stated conditions on scaling functions, $T$ is a contraction. Let $g_1, g_2 \in \mathcal{L}^p(I)$ and $1 \le p < \infty$.
\begin{equation*}
\begin{split}
\|Tg_1-Tg_2\|_p^p = &~ \int \limits_{I} \big|(Tg_1-Tg_2)(x)\big|^p ~\mathrm{d}x\\
=~& \underset {i \in \mathbb{N}_{N-1}}\sum \int \limits_{I_i} \Big|   \alpha_i \big(L_i^{-1}(x)\big)\big[g_1\big(L_i^{-1}(x)\big)-g_2\big(L_i^{-1}(x)\big)\big]\Big|^p ~\mathrm{d}x\\
=~& \underset {i \in \mathbb{N}_{N-1}}\sum a_i \int \limits_{I}\Big| \alpha_i(\tilde{x})\big[g_1(\tilde{x})-g_2(\tilde{x})\big]\Big|^p~\mathrm{d}\tilde{x}\\
\le~& \underset {i \in \mathbb{N}_{N-1}}\sum a_i \|\alpha_i\|_\infty^p \int \limits_{I} \big|g_1(\tilde{x})-g_2(\tilde{x})\big|^p ~\mathrm{d}\tilde{x}\\
=~& \Big[\underset {i \in \mathbb{N}_{N-1}}\sum a_i \|\alpha_i\|_\infty^p \Big] \|g_1-g_2\|_p^p.
\end{split}
\end{equation*}
In the third step of the preceding analysis we used the change of variable $\tilde{x}=L_i^{-1}(x)$, whereas the rest follows directly by computations.
Hence, we see that for $\Big[\sum \limits_{i \in \mathbb{N}_{N-1}} a_i \|\alpha_i\|_\infty^p\Big]^{1/p}<1$ the map $T$ is a contraction. Proof is similar for the remaining cases.
\end{proof}
Recall that, in general, the $\alpha$-fractal function has noninteger Hausdorff and Minkowski dimensions that depend on the scale vector $\alpha$ and hence the map $f \mapsto f^\alpha$ is a ``roughing" operation. In the following theorem we prove that by appropriate choices of scaling functions and base function, the order of continuity of $f$ can be preserved in $f^\alpha$. This extends the construction of smooth fractal function enunciated in \cite{NS1} with the aid of Barnsley-Harrington (BH) theorem \cite{BH}. It is to be stressed that BH theorem does not rescue us in the setting of variable scaling.
\begin{theorem}\label{SAINdiffthm}
Let $f \in \mathcal{C}^k(I)$, where $k \in \mathbb{N}$. Suppose that $\Delta=\{x_1,x_2,\dots, x_N\}$ is a partition of $I$ satisfying $x_1<x_2<\dots<x_N$, $I_i:=[x_i,x_{i+1}]$ for $i \in \mathbb{N}_{N-1}$ and $L_i:I \rightarrow I_i$ are affine maps $L_i(x)=a_i x+d_i$ satisfying $L_i(x_1)=x_i$ and $L_i(x_N)=x_{i+1}$ for $i\in \mathbb{N}_{N-1}$. Suppose that $k$-times continuously differentiable scaling functions and base function are selected so that
$$\|\alpha_i\|_{\mathcal{C}^k} \le (\frac{a_i}{2})^k; ~b^{(r)}(x_1)=f^{(r)}(x_1),~ b^{(r)}(x_N)= f^{(r)}(x_N), ~i \in \mathbb{N}_{N-1}, ~r \in \mathbb{N}_k \cup \{0\}.$$ Then the RB operator defined in (\ref{RB}) is a contraction on
the complete metric space $$\mathcal{C}_{f}^k(I) := \big\{g \in \mathcal{C}^k(I): g^{(r)}(x_1)=f^{(r)}(x_1),~g^{(r)}(x_N)=f^{(r)}(x_N),~ r \in \mathbb{N}_k \cup \{0\}\big\}.$$
 Furthermore, the derivative $(f^\alpha)^{(r)}$  of its unique fixed point $f^\alpha$ satisfies the self-referential equation
$$(f^ \alpha)^{(r)}(x)=f ^{(r)}(x)+ a_i^{-r} \Big[ \sum_{j=0}^{r} {r \choose j} \alpha_i^{(r-j)}\big(L_i^{-1}(x)\big) (f^ \alpha-b)^{(j)}\big(L_i^{-1}(x)\big)\Big], ~~ x \in I_i, ~i \in \mathbb{N}_{N-1},$$ and consequently, $f^\alpha$ agrees with $f$ at the knot points up to $k$-th derivative.
\end{theorem}
\begin{proof}
By the continuity conditions imposed on the maps $\alpha_i$ and $b$, it follows at once that $Tg$ is continuous on each subinterval $I_i$.
By the Leibnitz rule of differentiation we obtain
$$ (Tg)^{(r)} (x) = f^{(r)}(x) + a_i^{-r} \Big[ \sum_{j=0}^{r} {r \choose j} \alpha_i^{(r-j)}\big(L_i^{-1}(x)\big) (g-b)^{(j)}\big(L_i^{-1}(x)\big)\Big], \quad x \in I_i, \quad i \in \mathbb{N}_{N-1}.$$
Using  $L_i^{-1}(x_i)=x_1$ and $L_{i-1}^{-1}(x_i)=x_N$ we infer that
$$(Tg)^{(r)}(x_i^+) = (Tg)^{(r)} (x_i^-)=f^{(r)}(x_i), ~i=2,3,\dots, N-1,~ r \in \mathbb{N}_k \cup \{0\}.$$
Further, it is plain to see that $$(Tg)^{(r)}(x_1)= f^{(r)}(x_1), \quad (Tg)^{(r)}(x_N)= f^{(r)}(x_N) \quad \text{for} \quad r \in \mathbb{N}_k \cup \{0\}.$$
Consequently, $Tg \in \mathcal{C}_{f}^k(I)$ and $T$ is a self-map.\\
We will now demonstrate that  $T$ is a contraction map. For $g_1,g_2 \in \mathcal{C}_f^k(I)$, we obtain
$$ \big|(Tg_1)^{(r)} (x) - (Tg_2)^{(r)}(x) \big| \le a_i^{-r}2^r \|\alpha_i\|_{\mathcal{C}^r} \|g_1-g_2\|_{\mathcal{C}^r},\quad  x \in I_i, \quad i \in \mathbb{N}_{N-1}, \quad r\in \mathbb{N}_k \cup \{0\}.$$
The previous inequality implies
$$\|Tg_1 -Tg_2\|_{\mathcal{C}^k} \le \max \Big\{(\frac{2}{a_i})^k \|\alpha_i\|_{\mathcal{C}^k}: i \in \mathbb{N}_{N-1}\Big\} \|g_1-g_2\|_{\mathcal{C}^k}.$$
The assumption on the scaling functions now ensure that $T$ is a contraction map, and, hence by the Banach fixed point theorem, $T$ has a unique  fixed point $f^\alpha$. Further, it follows that $(f^\alpha)^{(r)}$ obeys the functional equation given in the statement and  $(f^\alpha)^{(r)}(x_i)=f^{(r)}(x_i)$ for $r\in \mathbb{N}_k \cup \{0\}$ and $i \in \mathbb{N}_N$, completing the proof.
\end{proof}
To ease the exposition, we assume constant scaling function $\alpha_i(x)=\alpha_i$ for all $x \in I$ in the following theorem. However, we note in passing that one can handle the setting of scaling functions in a similar manner by using the following Leibnitz formula for weak derivatives which states \cite{EV}:\\
Assume $u \in \mathcal{W}^{k,p}(I)$ and $\varphi$ is infinitely differentiable. Then $\varphi u \in \mathcal{W}^{k,p}(I)$ and for $j \le k$
$$D^j (\varphi u) = \sum_{r=0}^{j} {j \choose r}D^r \varphi D^{j-r} u.$$
To prepare for the next theorem let us recall the chain rule for weak derivative adapted to our situation: \\
Let $u \in \mathcal{W}^{1,1}(I)$ and $\varphi: V \subset \mathbb{R} \to I$ be diffeomorphism such that  $\varphi$ and $(\varphi^{-1})'$ are bounded. Then $u \circ \varphi \in \mathcal{W}^{1,1}(V)$ and $D (u \circ \varphi) = (Du \circ \varphi). D \varphi$.
\begin{theorem}\label{FFSS}
Let $f \in \mathcal{W}^{k,p}(I)$. Suppose that $\Delta=\{x_1,x_2,\dots, x_N\}$ is a partition of~ $I$ with increasing abscissae, $I_i:=[x_i,x_{i+1})$ for $i \in \mathbb{N}_{N-2}$ and $I_{N-1}:=[x_{N-1}, x_N]$.
Let $L_i:I \rightarrow I_i$ be  affine maps $L_i(x)=a_i x+d_i$ satisfying $L_i(x_1)=x_i$ and $L_i(x_N^-)=x_{i+1}$ for $i\in \mathbb{N}_{N-1}$.
Assume that $b \in \mathcal{W}^{k,p}(I)$ and the scaling factors satisfy
\begin{equation*}
\left\{\begin{array}{rllll}
& \Big[\sum \limits_{i \in \mathbb{N}_{N-1}} \frac{|\alpha_i|^p}{a_i^{kp-1}} \Big]^{1/p}<1 , \; \text{for} ~~ p \in [1, \infty),\\
& \underset{i \in \mathbb{N}_{N-1}}\max~ \{\frac{|\alpha_i|}{a_i^k}\} <1 , \; \; \; \;\; \;\;\;\;\;\;\text{for} ~~p =\infty.
\end{array}\right.
\end{equation*}
Then the RB operator (\ref{RBCS}) is a contraction map  on $\mathcal{W}^{k,p}(I)$ and the unique fixed point $f^\alpha \in \mathcal{W}^{k,p}(I)$ obeys the self-referential equation (\ref{ICAFWeq4}).
\end{theorem}
\begin{proof}
From the stated assumptions on the scaling factors and base function it follows that $T$ is a well-defined self-map on $\mathcal{W}^{k,p}(I)$. Next we shall show that $T$ is a contraction. \\
Following the proof of Theorem \ref{RRAFIFthm1}, for $1 \le p <\infty$, we obtain
$$\|Tg_1-Tg_2\|_p \le \Big[\sum_{i \in \mathbb{N}_{N-1}} a_i |\alpha_i|^p\Big]^{\frac{1}{p}} \|g_1-g_2\|_p.$$
For $1\le j \le k$, straightforward computations yield
\begin{equation*}
\begin{split}
\|D^j(Tg_1)-D^j(Tg_2)\|_p ^p = &~ \int_{I} \big|D^j(Tg_1) (x)- D^j(Tg_2) (x)\big|^p ~\mathrm{d}x\\
=&~ \sum_{i \in \mathbb{N}_{N-1}} \int_{I_i} \big|D^j[\alpha_i (g_1-g_2)\circ L_i^{-1}(x)]\big|^p  ~\mathrm{d}x\\
=& ~\sum_{i \in \mathbb{N}_{N-1}} |\alpha_i|^p \frac{1}{a_i^{jp}} \int_{I_i}\big|D^j(g_1-g_2) \big(L_i^{-1}(x)\big)\big|^p  ~\mathrm{d}x\\
=& ~\sum_{i \in \mathbb{N}_{N-1}} |\alpha_i|^p \frac{1}{a_i^{jp-1}}\int_{I} \big|D^j(g_1-g_2)(\tilde{x})\big|^p ~\mathrm{d}\tilde{x}\\
=&~\Big[ \sum_{i \in \mathbb{N}_{N-1}} |\alpha_i|^p \frac{1}{a_i^{jp-1}} \Big]\|D^j(g_1-g_2)\|_p^p,
\end{split}
\end{equation*}
which implies that for $1 \le j \le k$,
$$\|D^j(Tg_1)-D^j(Tg_2)\|_p = \Big[ \sum_{i \in \mathbb{N}_{N-1}}  \frac{|\alpha_i|^p}{a_i^{jp-1}} \Big]^{\frac{1}{p}} \|D^j(g_1-g_2)\|_p.$$
Consequently, for $1 \le p < \infty$,
\begin{equation*}
\begin{split}
\|Tg_1-Tg_2\|_{\mathcal{W}^{k,p}} =&~ \Big[\sum_{j=0}^k  \|D^j(Tg_1-Tg_2)\|_p^p\Big]^{\frac{1}{p}} \\
=&~\Bigg [\sum_{j=0}^k  \Big[ \sum_{i \in \mathbb{N}_{N-1}}\frac{|\alpha_i|^p}{a_i^{jp-1}} \Big]\|D^j(g_1-g_2)\|_p^p\Bigg]^{\frac{1}{p}}.
\end{split}
\end{equation*}
Bearing in mind that $0<a_i<1$, $1 \le p < \infty$, and $k \ge1$, we have $$ \frac{1}{a_i^{jp-1}} \le \frac{1}{a_i ^{kp-1}} \quad   \text{for}\quad  0 \le j \le k.$$ Hence,
\begin{equation*}
\begin{split}
\|Tg_1-Tg_2\|_{\mathcal{W}^{k,p}} \le &~  \Big[\sum_{i \in \mathbb{N}_{N-1}}\frac{|\alpha_i|^p}{a_i^{kp-1}} \Big]^{\frac{1}{p}} \Bigg[\sum_{j=0}^k \|D^j(g_1-g_2)\|_p^p\Bigg]^{\frac{1}{p}}.\\
= &~ \Big[\sum_{i \in \mathbb{N}_{N-1}}\frac{|\alpha_i|^p}{a_i^{kp-1}} \Big]^{\frac{1}{p}}\|g_1-g_2\|_{\mathcal{W}^{k,p}}
\end{split}
\end{equation*}
Therefore, for $1 \le p < \infty$, the stated assumption on scaling factors ensures that $T$ is a contraction on $\mathcal{W}^{k,p}(I)$.\\
Next consider the case $p= \infty$. For $x \in I_i$, $i \in \mathbb{N}_{N-1}$
\begin{equation*}
\begin{split}
\big|(Tg_1)(x)-(Tg_2) (x)\big| = &~\big|\alpha_i (g_1-g_2)\big(L_i^{-1}(x)\big)\big|\\
\le &~|\alpha_i| \|g_1-g_2\|_\infty,
\end{split}
\end{equation*}
from which one can read
$$\|Tg_1-Tg_2\|_\infty \le \max \{|\alpha_i|:i \in \mathbb{N}_{N-1}\} \|g_1 -g_2\|_\infty.$$
Similarly, for $1 \le j \le k$,
\begin{equation*}
\begin{split}
\big|(D^j Tg_1) (x) - (D^j Tg_2) (x) \big| =&~ \big|\alpha_i D^j \big((g_1-g_2) \circ L_i^{-1} (x)\big)\big | \\
=&~ |\alpha_i| \frac{1}{a_i^j} \big|D^j(g_1-g_2) \circ L_i^{-1} (x)\big|\\
\le &~ \frac{|\alpha_i|}{a_i^j} \|D^j(g_1-g_2)\|_\infty,
\end{split}
\end{equation*}
which asserts
$$\|D^j(Tg_1)-D^j(Tg_2)\|_ \infty \le \max \big\{\frac{|\alpha_i|}{a_i^j}: i\in \mathbb{N}_{N-1}\big\} \|D^j(g_1-g_2)\|_\infty.$$
Thus we have
\begin{equation*}
\begin{split}
\|Tg_1-Tg_2\|_{\mathcal{W}^{k,\infty}} =&~ \sum_{j=0}^k \|D^j(Tg_1-Tg_2)\|_\infty \\
\le &~ \max \Big \{ \max \big\{\frac{|\alpha_i|}{a_i^j}: i\in \mathbb{N}_{N-1}\big\}: j \in \mathbb{N}_k \cup \{0\}\Big\} \|g_1-g_2\|_{\mathcal{W}^{k,p}}\\
=&~\max \Big\{\frac{|\alpha_i|}{a_i^k}: i \in \mathbb{N}_{N-1}\Big\} \|g_1-g_2\|_{\mathcal{W}^{k,p}},
\end{split}
\end{equation*}
which with prescribed conditions on the scaling factors implies that $T$ is a contraction.
\end{proof}
\begin{theorem}\label{FFSH}
Let $ f\in \mathcal{C}^{k,\sigma}(I)$. Suppose that $\Delta=\{x_1,x_2,\dots, x_N\}$ is a partition of~ $I$ satisfying $x_1<x_2<\dots<x_N$, $I_i:=[x_i,x_{i+1}]$ for $i \in \mathbb{N}_{N-1}$ and $L_i:I \rightarrow I_i$ are affine maps $L_i(x)=a_i x+d_i$ satisfying $L_i(x_1)=x_i$ and $L_i(x_N)=x_{i+1}$ for $i\in \mathbb{N}_{N-1}$. Let the base function $b \in \mathcal{C}^{k,\sigma}(I)$ and the scaling factors $\alpha_i$ satisfy
$$\max\Big\{\frac{|\alpha_i|}{a_i^{\sigma+k}}: i \in \mathbb{N}_{N-1}\Big\} <1,$$
$$ b^{(r)}(x_1)=f^{(r)}(x_1),~ b^{(r)}(x_N)= f^{(r)}(x_N), ~r \in \mathbb{N}_k \cup \{0\}.$$
Then the RB operator $T$ defined in (\ref{RBCS}) is a contraction map  on $\mathcal{C}_f^{k,\sigma}(I) \subset \mathcal{C}^{k,\sigma}(I)$ defined by
$$\mathcal{C}_f^{k,\sigma}(I)= \big\{ g \in \mathcal{C}^{k,\sigma}(I): g^{(r)}(x_1) = f^{(r)}(x_1),~ g^{(r)}(x_N) = f^{(r)}(x_N),~ 0 \le r \le k\big\}$$
and the unique fixed point $f^\alpha \in \mathcal{C}^{k,\sigma}(I)$ obeys the self-referential equation (\ref{ICAFWeq4}).
\end{theorem}
\begin{proof}
Following the proof of Theorem \ref{SAINdiffthm} we assert that $Tg \in \mathcal{C}^{k}(I)$ whenever $g \in \mathcal{C}^{k,\sigma}(I)$. Further,
\begin{equation} \label{new1}
(Tg)^{(r)} (x) = f^{(r)}(x) + \frac{\alpha_i}{a_i^r} (g-b)^{(r)} \big(L_i^{-1}(x)\big), ~x \in I_i, ~i \in \mathbb{N}_{N-1}, r\in \mathbb{N}_k \cup \{0\}.
\end{equation}
The $\sigma$-th H\"older seminorm of $(Tg)^{(k)}$ is given by
\begin{equation*}
\begin{split}
[(Tg)^{(k)}]_\sigma =&~ \sup_{x, y \in I, x \neq y} \frac{|(Tg)^{(k)}(x)-(Tg)^{(k)}(y)|}{|x-y|^\sigma}\\
=&~ \max _{i \in \mathbb{N}_{N-1}} \sup _{x, y \in I_i, x \neq y}\frac{\big|\frac{\alpha_i}{a_i^k}[(g-b)^{(k)} \big(L_i^{-1}(x)\big)-(g-b)^{(k)} \big(L_i^{-1}(y)\big)]\big|}{|x-y|^\sigma}\\
%\end{split}
%\end{equation*}
%\begin{equation*}
%\begin{split}
\le &~ \max _{i \in \mathbb{N}_{N-1}}(\frac{|\alpha_i|}{a_i^k})\sup _{x, y \in I_i, x \neq y} \Big[ \frac{g^{(k)}\big(L_i^{-1}(x)\big)-g^{(k)}\big(L_i^{-1}(y)\big)}{|x-y|^\sigma}+  \frac{b^{(k)}\big(L_i^{-1}(x))-b^{(k)}\big (L_i^{-1}(y)\big)}{|x-y|^\sigma}\Big]\\
 =&~ \max _{i \in \mathbb{N}_{N-1}}(\frac{|\alpha_i|}{a_i^k})\sup _{x, y \in I_i, x \neq y} \Big[ \frac{|g^{(k)}\big(L_i^{-1}(x)\big)-g^{(k)}\big(L_i^{-1}(y)\big)|}{a_i ^\sigma |L_i^{-1}(x)-L_i^{-1}(y)|^\sigma}+  \frac{|b^{(k)}\big(L_i^{-1}(x))-b^{(k)}\big (L_i^{-1}(y)\big)|}{a_i ^\sigma |L_i^{-1}(x)-L_i^{-1}(y)|^\sigma}\Big]\\
 = &~ \max _{i \in \mathbb{N}_{N-1}}(\frac{|\alpha_i|}{a_i^{k+\sigma}}) \sup _{\tilde{x}, \tilde{y}\in I, \tilde{x} \neq \tilde{y}}\Big[\frac{|g^{(k)}(\tilde{x})-g^{(k)}(\tilde{y})|}{|\tilde{x}-\tilde{y}|^\sigma}+\frac{|b^{(k)}(\tilde{x})-b^{(k)}(\tilde{y})|}{|\tilde{x}-\tilde{y}|^\sigma}\Big]\\
=& \max _{i \in \mathbb{N}_{N-1}}(\frac{|\alpha_i|}{a_i^{k+\sigma}})\big( [g^{(k)}]_\sigma + [b^{(k)}]_\sigma \big).
\end{split}
\end{equation*}
Since $b$ and $g$ are in $\mathcal{C}^{k, \sigma}(I)$, the previous estimate ensures that $[(Tg)^{(k)}]_\sigma < \infty$ and hence that $Tg \in \mathcal{C}^{k, \sigma} (I)$.
Again from (\ref{new1}), for $x \in I_i$,
\begin{equation*}
\begin{split}
\big|(Tg_1-Tg_2)^{(r)}(x)\big| =&~ \frac{|\alpha_i|}{a_i^r} \big|(g_1-g_2)^{(r)}\big(L_i^{-1}(x) \big)\big|\\
\le &~ \frac{|\alpha_i|}{a_i^r} \|(g_1-g_2)^{(r)}\|_\infty
\end{split}
\end{equation*}
and hence one readily obtains
$$ \|(Tg_1-Tg_2)^{(r)}\|_\infty \le \max \Big\{\frac{|\alpha_i|}{a_i^r}: i \in \mathbb{N}_{N-1}\Big \}~ \|(g_1-g_2)^{(r)}\|_\infty.$$
On lines similar to the estimation of $[(Tg)^{(k)}]_\sigma$, we get
$$[(Tg_1-Tg_2)^{(k)}]_\sigma \le \max\Big\{\frac{|\alpha_i|}{a_i^{k+\sigma}}: i \in \mathbb{N}_{N-1}\Big\}[(g_1-g_2)^{(k)}]_\sigma.$$
From these computations we gather that
\begin{equation*}
\begin{split}
\|Tg_1-Tg_2\|_{\mathcal{C}^{k,\sigma}} = &~ \sum_{r=0}^k \|(Tg_1-Tg_2)^{(r)}\|_\infty+ [(Tg_1-Tg_2)^{(k)}]_\sigma\\
\le &~ \sum_{r=0}^k\max \Big\{\frac{|\alpha_i|}{a_i^r}: i \in \mathbb{N}_{N-1}\Big \} \|(g_1-g_2)^{(r)}\|_\infty + \max\Big\{\frac{|\alpha_i|}{a_i^{k+\sigma}}: i \in \mathbb{N}_{N-1}\Big\}[(g_1-g_2)^{(k)}]_\sigma\\
\le &~ \max _{r \in \mathbb{N}_k \cup \{0\}} \max _{i \in \mathbb{N}_{N-1}} \big\{\frac{|\alpha_i|}{a_i^r}\big\} \sum_{r=0}^k \|(g_1-g_2)^{(r)}\|_\infty + \max_{i \in \mathbb{N}_{N-1}}\Big\{\frac{|\alpha_i|}{a_i^{k+\sigma}}\Big\}[(g_1-g_2)^{(k)}]_\sigma\\
\le &~\max\big\{\frac{|\alpha_i|}{a_i^{k+\sigma}}: i \in \mathbb{N}_{N-1}\big\} \Big[ \sum_{r=0}^k \|(g_1-g_2)^{(r)}\|_\infty+[(g_1-g_2)^{(k)}]_\sigma\Big]\\
=&~\max\big\{\frac{|\alpha_i|}{a_i^{k+\sigma}}: i \in \mathbb{N}_{N-1}\big\} \|g_1-g_2\|_{\mathcal{C}^{k,\sigma}}.
\end{split}
\end{equation*}
The assumption on scaling factors now yields that $T$ is a contraction on $\mathcal{C}_f^{k, \sigma} (I)$, and hence the proof.
\end{proof}
\section{Fractal operator on Function Spaces}\label{FFS4}
Theorems \ref{FFSB}-\ref{FFSH} established in the previous section illustrate, albeit indirectly, the existence of an operator that assigns a function $f$ to its self-referential (fractal) analogue $f^\alpha$. To be precise, for a fixed set of scaling functions (factors) $\alpha_i$ and for suitable choice of base function,  there exists a map
$$\mathcal{F}^\alpha: X \to X; \quad  \mathcal{F}^\alpha(f) = f^\alpha,$$
where $X$ is one of the Banach spaces $\mathcal{B}(I)$, $\mathcal{C}^k(I)$, $\mathcal{L}^p(I)$, $\mathcal{W}^{k,p}(I)$, or $\mathcal{C}^{k, \sigma}(I)$. In this section, we offer a couple of fundamental results on this fractal operator $\mathcal{F}^\alpha$. For definiteness, we shall work with $X= \mathcal{W}^{k,p}(I)$,  other cases can be similarly dealt with. Further, we assume that the base function $b$ depends on $f$ linearly and $b = Lf$, where $L: X \to X$ is a bounded linear operator  with
respect to the norm in the corresponding space.
\begin{proposition}\label{Prop1}
For $f \in \mathcal{W}^{k,p}(I)$ and scaling factors satisfying conditions prescribed in Theorem \ref{FFSS}, we have
\begin{equation*}
\|f^\alpha-f\|_{\mathcal{W}^{k,p}} \le \left\{\begin{array}{rllll}
& \big[\sum \limits_{i \in \mathbb{N}_{N-1}} \frac{|\alpha_i|^p}{a_i^{kp-1}}\big]^{\frac{1}{p}} ~ \|f^\alpha-b\|_{\mathcal{W}^{k,p}} , \; \text{for} ~~ p \in [1, \infty),\\
& \max \big\{\frac{|\alpha_i|}{a_i^k}:i \in \mathbb{N}_{N-1}\big\} ~\|f^\alpha-b\|_{\mathcal{W}^{k,p}} , \; \; \text{for} ~~p =\infty.
\end{array}\right.
\end{equation*}
\end{proposition}
\begin{proof}
We have the functional equation
$$f^\alpha(x) = f(x) + \alpha_i (f^\alpha-b) \big(L_i^{-1}(x) \big), \quad  x \in I_i, \quad  i \in \mathbb{N}_{N-1}.$$
Therefore,
$$\big(D^k (f^\alpha-f)\big)(x)= \frac{\alpha_i}{a_i^k} (D^k (f^\alpha-b))\big(L_i^{-1}(x)\big), \quad  x \in I_i, \quad  i \in \mathbb{N}_{N-1}.$$
Assume that $1 \le p <\infty$. By a series of self-evident steps
\begin{equation*}
\begin{split}
\|D^k(f^\alpha-f)\|_p^p=&~ \int_{I} |D^k(f^\alpha-f)(x)|^p ~\mathrm{d}x\\
=&~ \sum_{i \in \mathbb{N}_{N-1}} \big(\frac{|\alpha_i|}{a_i^k}\big)^p \int_{I_i} \big|D^k (f^\alpha-b)\big(L_i^{-1}(x)\big)\big|^p ~\mathrm{d}x\\
=&~\sum_{i \in \mathbb{N}_{N-1}} \big(\frac{|\alpha_i|}{a_i^k}\big)^p \int_{I} a_i\big|D^k (f^\alpha-b)(\tilde{x})\big|^p ~\mathrm{d}\tilde{x}\\
=&~\sum_{i \in \mathbb{N}_{N-1}} \frac{|\alpha_i|^p}{a_i^{kp-1}} \|D^k(f^\alpha-b)\|_p^p,
\end{split}
\end{equation*}
and consequently,
$$\|D^k(f^\alpha-f)\|_p \le \Big[\sum_{i \in \mathbb{N}_{N-1}} \frac{|\alpha_i|^p}{a_i^{kp-1}}\Big]^{\frac{1}{p}}~ \|D^k(f^\alpha-b)\|_p.$$
Thus
\begin{equation*}
\begin{split}
\|f^\alpha-f\|_{\mathcal{W}^{k,p}} =&~ \|f^\alpha-f\|_p + \|D^k(f^\alpha-f)\|_p\\
\le &~ \Big[\sum_{i \in \mathbb{N}_{N-1}} a_i |\alpha_i|^p\Big]^{\frac{1}{p}} \|f^\alpha-b\|_p + \Big[\sum_{i \in \mathbb{N}_{N-1}} \frac{|\alpha_i|^p}{a_i^{kp-1}}\Big]^{\frac{1}{p}}~ \|D^k(f^\alpha-b)\|_p\\
\le &~ \max \Big\{ \Big[\sum_{i \in \mathbb{N}_{N-1}} a_i |\alpha_i|^p\Big]^{\frac{1}{p}},  \Big[\sum_{i \in \mathbb{N}_{N-1}} \frac{|\alpha_i|^p}{a_i^{kp-1}}\Big]^{\frac{1}{p}} \Big\} \|f^\alpha-b\|_{\mathcal{W}^{k,p}}\\
= &~\Big[\sum_{i \in \mathbb{N}_{N-1}} \frac{|\alpha_i|^p}{a_i^{kp-1}}\Big]^{\frac{1}{p}} \|f^\alpha-b\|_{\mathcal{W}^{k,p}}.
\end{split}
\end{equation*}
The proof for $p=\infty$ is similar.
\end{proof}
\begin{theorem}\label{Thm1}
Let $I_d$ be the identity operator on $\mathcal{W}^{k,p}(I)$, $b=Lf$, where $L$ is a bounded linear operator on $\mathcal{W}^{k,p}(I)$, and the scaling factors satisfy conditions prescribed in Theorem \ref{FFSS}. Then
$$\|f^\alpha-f\|_{\mathcal{W}^{k,p}} \le \frac{\big[\sum \limits_{i \in \mathbb{N}_{N-1}} \frac{|\alpha_i|^p}{a_i^{kp-1}}\big]^{\frac{1}{p}}} {1-\big[\sum \limits_{i \in \mathbb{N}_{N-1}} \frac{|\alpha_i|^p}{a_i^{kp-1}}\big]^{\frac{1}{p}}}\|I_d-L\|~\|f\|_{\mathcal{W}^{k,p}},\quad \text{for} \quad 1 \le p < \infty,$$
$$\|f^\alpha-f\|_{\mathcal{W}^{k,p}} \le \frac{ \max \big\{\frac{|\alpha_i|}{a_i^k}:i \in \mathbb{N}_{N-1}\big\}}{1- \max \big\{\frac{|\alpha_i|}{a_i^k}:i \in \mathbb{N}_{N-1}\big\}}\|I_d-L\|~\|f\|_{\mathcal{W}^{k,p}}, \quad{for} \quad p =\infty. $$
\end{theorem}
\begin{proof}
The previous  proposition in conjunction with the triangle inequality yields
\begin{equation*}
\begin{split}
\|f^\alpha-f\|_{\mathcal{W}^{k,p}} \le&~ K \|f^\alpha-b\|_{\mathcal{W}^{k,p}}\\
\le &~ K \big[\|f^\alpha-f\|_\infty + \|f- Lf\|_\infty\big]\\
%\end{split}
%\end{equation*}
%\begin{equation*}
%\begin{split}
\le &~K \big[\|f^\alpha-f\|_\infty + \|I_d-L\| \|f\|_\infty \big],
\end{split}
\end{equation*}
where
\begin{equation*}
K=\left\{\begin{array}{rllll}
&\big[\sum \limits_{i \in \mathbb{N}_{N-1}} \frac{|\alpha_i|^p}{a_i^{kp-1}}\big]^{\frac{1}{p}}  , \; \text{for} ~~ p \in [1, \infty),\\
& \max \big\{\frac{|\alpha_i|}{a_i^k}:i \in \mathbb{N}_{N-1}\big\}, \; \;  \;\text{for} ~~p =\infty.\\
\end{array}\right.
\end{equation*}
from which the result follows at once.
\end{proof}
\begin{theorem}\label{THM2}
For scaling factors satisfying conditions prescribed in Theorem \ref{FFSS}, the self-referential operator $\mathcal{F}^\alpha: \mathcal{W}^{k,p}(I)\to \mathcal{W}^{k,p}(I)$ is a bounded linear operator which reduces to the identity operator for $\alpha=0$.
\end{theorem}
\begin{proof}
Let $f_1$ and $f_2$ be in $\mathcal{W}^{k,p}(I)$ and $\beta_1$, $\beta_2$ reals. The functional equation for $f_1^\alpha= \mathcal{F}^\alpha(f_1)$ and $f_2^\alpha= \mathcal{F}^\alpha(f_2)$ are given by
$$f_1^\alpha(x) = f_1(x) + \alpha_i (f_1^\alpha-Lf_1)\big(L_i^{-1}(x)\big), ~ x \in I_i, ~i \in \mathbb{N}_{N-1},$$
$$f_2^\alpha(x) = f_2(x) + \alpha_i (f_2^\alpha-Lf_2)\big(L_i^{-1}(x)\big), ~ x \in I_i, ~i \in \mathbb{N}_{N-1}.$$
Therefore
$$(\beta_1 f_1^\alpha + \beta_2 f_2^\alpha)(x)= (\beta_1 f_1 + \beta_2 f_2)(x)+ \alpha_i \big[\beta_1 f_1^\alpha+ \beta_2 f_2^\alpha- L(\beta_1 f_1+\beta_2 f_2)\big]\big(L_i^{-1}(x)\big),$$
from which it follows that $\beta_1f_1^\alpha + \beta_2 f_2^\alpha$ is a fixed point of the operator
$$(Tg)(x) = (\beta_1 f_1 + \beta_2 f_2)(x) + \alpha_i \big(g-L(\beta_1 f_1+\beta_2 f_2)\big)\big(L_i^{-1}(x)\big)$$
By the uniqueness of the fixed point we see that
$$\mathcal{F}^\alpha(\beta_1 f_1 + \beta_2 f_2)=(\beta_1 f_1 + \beta_2 f_2)^\alpha= \beta_1 f_1^\alpha + \beta_2 f_2^\alpha=\beta_1 \mathcal{F}^\alpha(f_1)+\beta_2 \mathcal{F}^\alpha(f_2),$$
proving the linearity of $\mathcal{F}^\alpha$.
In view of  the previous theorem we have
\begin{equation*}
\begin{split}
\|\mathcal{F}^\alpha(f)\|_{\mathcal{W}^{k,p}}= \|f^\alpha\|_{\mathcal{W}^{k,p}} \le &~ \|f^\alpha-f\|_{\mathcal{W}^{k,p}}+ \|f\|_{\mathcal{W}^{k,p}}\\
\le &~ \frac{K}{1-K}\|I_d-L\|~\|f\|_{\mathcal{W}^{k,p}} + \|f\|_{\mathcal{W}^{k,p}}\\
=&~ \Big[1+  \frac{K}{1-K}\|I_d-L\|~  \Big]\|f\|_{\mathcal{W}^{k,p}},
\end{split}
\end{equation*}
which ensures that $\mathcal{F}^\alpha$ is bounded and
$$\|\mathcal{F}^\alpha\| \le 1+  \frac{K}{1-K}\|I_d-L\|.$$
The last part of the theorem follows by noting that for $\alpha=0$, $f^\alpha=f$.
\end{proof}
\begin{theorem}\label{THM3}
Consider a scale vector $\alpha \in \mathbb{R}^{N-1}$ whose components satisfy
\begin{equation*}
\left\{\begin{array}{rllll}
& \Big[\sum \limits_{i \in \mathbb{N}_{N-1}} \frac{|\alpha_i|^p}{a_i^{kp-1}} \Big]^{1/p}<\min \big\{1, \|L\|^{-1}\big\} , \; \text{for} ~~ p \in [1, \infty),\\
& \underset{i \in \mathbb{N}_{N-1}}\max~ \{\frac{|\alpha_i|}{a_i^k}\} <\min\big\{1, \|L\|^{-1}\big\} , \; \; \; \;\; \;\;\;\;\;\;\text{for} ~~p =\infty.
\end{array}\right.
\end{equation*}
The corresponding fractal operator $\mathcal{F}^\alpha$ is bounded below. In particular,
$\mathcal{F}^\alpha$ is injective and has a closed range. In fact, $\mathcal{F}^\alpha: \mathcal{W}^{k,p}(I) \to \mathcal{F}^\alpha \big( \mathcal{W}^{k,p}(I)\big)$ is a topological isomorphism.
\end{theorem}
\begin{proof}
From Proposition \ref{Prop1}
$$\|f\|_{\mathcal{W}^{k,p}}-\|f^\alpha\|_{\mathcal{W}^{k,p}}\le \|f-f^\alpha\|_{\mathcal{W}^{k,p}}\le K \|f^\alpha-Lf\|_{\mathcal{W}^{k,p}}\le K \big[\|f^\alpha\|_{\mathcal{W}^{k,p}}+\|L\|\|f\|_{\mathcal{W}^{k,p}}\big],$$
where $K$ is as prescribed in Theorem \ref{Thm1}. Also by the stated assumption on scale vector we have $K<\|L\|^{-1}$. Thus
$$ \|f\|_{\mathcal{W}^{k,p}} \le \frac{1+K}{1-K\|L\|} \|f^\alpha\|_{\mathcal{W}^{k,p}}.$$
That is, the operator $\mathcal{F}^\alpha$ is bounded below and hence, in particular, injective. To prove that $\mathcal{F}^\alpha \big(\mathcal{W}^{k,p}(I)\big)$ is closed, let $f_n^\alpha$ be a sequence in $\mathcal{F}^\alpha \big(\mathcal{W}^{k,p}(I)\big)$ such that
$f_n^\alpha \to g$. In particular,  $\big\{f_n^\alpha=\mathcal{F}^\alpha(f_n)\big\}$ is a Cauchy sequence in $\mathcal{F}^\alpha \big(\mathcal{W}^{k,p}(I)\big)$. Since
$$ \| f_n -f_m\|_{\mathcal{W}^{k,p}} \le  \frac{1+K}{1-K\|L\|} \| f_n^\alpha -f_m ^\alpha\|_{\mathcal{W}^{k,p}},$$
it follows that $\{f_n\}$ is a Cauchy sequence in $\mathcal{W}^{k,p}(I)$. Since $\mathcal{W}^{k,p}(I)$ is a complete space, there exists an $f \in \mathcal{W}^{k,p}(I)$ such that $f_n \to f$ and consequently, by the continuity of  the fractal operator $\mathcal{F}^\alpha$ we conclude that $g=\mathcal{F}^\alpha(f)=f^\alpha$. Now from the bounded inverse theorem (see, for instance, \cite{JC}) it follows that the inverse of the map $\mathcal{F}^\alpha: \mathcal{W}^{k,p}(I) \to \mathcal{F}^\alpha \big( \mathcal{W}^{k,p}(I)\big)$ is  a bounded linear operator, completing the proof.
\end{proof}
\begin{theorem}\label{THM6}
For a scale vector $\alpha \in \mathbb{R}^{N-1}$ whose components satisfy
 \begin{equation*}
\left\{\begin{array}{rllll}
& \Big[\sum \limits_{i \in \mathbb{N}_{N-1}} \frac{|\alpha_i|^p}{a_i^{kp-1}} \Big]^{1/p}<\big(1+\|I_d-L\|\big)^{-1}  , \; \text{for} ~~ p \in [1, \infty),\\
& \underset{i \in \mathbb{N}_{N-1}}\max~ \{\frac{|\alpha_i|}{a_i^k}\} <\big(1+\|I_d-L\|\big)^{-1}  , \; \; \; \;\; \;\;\;\;\;\;\text{for} ~~p =\infty,
\end{array}\right.
\end{equation*}
the fractal operator $\mathcal{F}^\alpha$ is a topological automorphism on $\mathcal{W}^{k,p}(I)$. Furthermore, with $K$ as in Theorem \ref{Thm1}
$$ \frac{1-K\|L\|}{1+K} \|f\|_{\mathcal{W}^{k,p}} \le \|\mathcal{F}^\alpha(f)\|_{\mathcal{W}^{k,p}} \le 1+  \frac{K}{1-K}\|I_d-L\| \|f\|_{\mathcal{W}^{k,p}}.$$
\end{theorem}
\begin{proof}
From Theorem \ref{Thm1} it can be easily seen that
$$\|I_d -\mathcal{F}^\alpha\| \le \frac {K}{1-K} \|I_d-L\|.$$
In view of the assumption on the scaling vector $\alpha$, we get $K<\big(1+\|I_d-L\|\big)^{-1}$ and hence $\|I_d -\mathcal{F}^\alpha\|<1$. Consequently, the Neumann series $\sum_{j=0}^\infty (I_d-\mathcal{F}^\alpha)^j$ is convergent in the operator norm and $\mathcal{F}^\alpha=I_d-(I_d-\mathcal{F}^\alpha)$ is invertible (see, for instance, \cite{JC}). Bearing in mind that
$$\|L \| = \big \| I_d- (I_d-L) \| \le 1 + \|I_d -L \|,$$
the bounds on $\|\mathcal{F}^\alpha(f)\|_{\mathcal{W}^{k,p}}$ follow from Theorem \ref{THM2} and \ref{THM3}.
\end{proof}
The existence of Schauder bases for
Sobolev spaces  is desired for demonstrating  the existence of solutions
of nonlinear boundary value problems. Sometimes it is interesting to look for the global structure involved
in a given problem and hence self-referentiality may be advantageous. Therefore, it is worth  to search for a Schauder basis for $\mathcal{W}^{k,p}(I)$ consisting of fractal functions. It is to this that we now turn. First let us recall the following theorem which asserts the existence of Schauder basis for
$\mathcal{W}^{k,p}(I)$, where $p \geq 1$ is a real number. We reproduce its proof here briefly for the sake of completeness.  Without loss of generality assume that $I=[0,1]$.
\begin{theorem} (Fu\v{c}\'{i}k).(see \cite{F} Theorem 4.7). For $1 \le p < \infty$,
$\mathcal{W}^{k,p}(I)$ has a Schauder basis.
\end{theorem}
\begin{proof}
Proof is by induction with respect to $k$. For $k=0$, $\mathcal{W}^{0,p}(I)=\mathcal{L}^p(I)$, and has a Schauder basis. Let $\{f_n^k\}$ be a Schauder basis for $\mathcal{W}^{k,p}(I)$ and $\{\beta_n^k\}$ be a sequence of continuous linear functionals such that for each $f\in \mathcal{W}^{k,p}(I)$, $f= \sum_{n=1}^\infty \beta_n^k (f) f_n^k$. For $f \in \mathcal{W}^{k,p}(I)$, set
\begin{equation*}
\begin{split}
& f_1^{k+1}(x) \equiv 1,~~~~~~\beta_1^{k+1}(f) =f(0)\\
&f_n^{k+1}(x) = \int_{0}^x f_{n-1}^k (t)~\mathrm{d}t,~~~~~~\beta_n^{k+1}(f) = \beta_{n-1}^k(f'), ~~n \ge 2.
\end{split}
\end{equation*}
Then $\{f_n^{k+1}\}$ is a Schauder basis in $\mathcal{W}^{k+1,p}(I)$.
\end{proof}
\begin{theorem}
For $1 \le p < \infty$, the space $\mathcal{W}^{k,p}(I)$ admits a Schauder basis consisting of self-referential  functions.
\end{theorem}
\begin{proof}
Let $\{f_n\}$ be a Schauder basis for $\mathcal{W}^{k,p}(I)$ with associated coefficient functionals $\{\beta_n\}$.
Consider the scaling factors $\alpha_i$, $i\in \mathbb{N}_{N-1}$, such that  the condition prescribed in Theorem \ref{THM6} is satisfied so that $\mathcal{F}^\alpha$ is a topological automorphism on $\mathcal{W}^{k,p}(I)$. Let $f \in \mathcal{W}^{k,p}(I)$. Then
$(\mathcal{F}^ \alpha)^{-1}(f) \in \mathcal{W}^{k,p}(I)$ and   $$(\mathcal{F}^\alpha)^{-1}(f)= \sum_{n=1}^\infty \beta_n \big(\big( \mathcal{F}^ \alpha)^{-1}(f)\big) f_n.$$
Since $\mathcal{F}^ \alpha$  is a bounded linear map we obtain  $$f= \sum_{n=1}^\infty  \beta_n\big(\big( \mathcal{F}^\alpha)^{-1}(f)\big) f_n^ \alpha.$$
Next we prove that the representation is unique. Let $f=
\sum_{n=1}^\infty \gamma_n f_n^ \alpha$ be another representation of $f$. Continuity of $(\mathcal{F}^ \alpha)^{-1}$
ensures that $(\mathcal{F}^ \alpha)^{-1}(f)=\sum_{n=1}^\infty \gamma_n f_n$
and hence  $\gamma_n =\beta_n\big((\mathcal{F}^\alpha)^{-1}(f)\big)$  for all $n \in \mathbb{N}$. Thus  $\{f_n^\alpha\}$ is a Schauder basis for $\mathcal{W}^{k,p}(I)$ comprising of self-referential functions.
\end{proof}
For $p=2$, the Sobolev space $\mathcal{W}^{k,2}(I)=\mathcal{H}^k(I)$ is a Hilbert space and hence we can talk about the adjoint (in the usual sense) $(\mathcal{F}^\alpha)^*$ of the fractal operator $\mathcal{F}^\alpha$. We have the following.
\begin{theorem}\label{THM5}
For a scale vector $\alpha \in \mathbb{R}^{N-1}$ satisfying $\Big[\sum_{i \in \mathbb{N}_{N-1}} \frac{|\alpha_i|^2}{a_i^{2k-1}}\Big]^{\frac{1}{2}}<\min \big\{1,\|L\|^{-1}\big\}$,
$$\mathcal{H}^k(I)=rg (\mathcal{F}^\alpha) \oplus ker \big(({\mathcal{F}^\alpha})^*\big),$$ where $rg (A)$ and $ker (A)$ represent range and kernel of an operator $A$. Also, $(\mathcal{F}^\alpha)^*$ is surjective.
\end{theorem}
\begin{proof}
From the proof of Theorem \ref{THM3} it follows that for a scale vector $\alpha$ satisfying the stated hypothesis, $rg (\mathcal{F}^\alpha)$ is closed. By the orthogonal decomposition theorem for a Hilbert space (see, for instance, \cite{JC})
$$\mathcal{H}^k(I)= rg (\mathcal{F}^\alpha) \oplus rg (\mathcal{F}^\alpha)^\bot = rg (\mathcal{F}^\alpha) \oplus ker  \big((\mathcal{F}^\alpha)^*\big).$$
Again by the orthogonal decomposition,
$$\mathcal{H}^k(I)=\overline{rg \big((\mathcal{F}^\alpha)^*\big)} \oplus \overline{rg \big((\mathcal{F}^\alpha)^*\big)}^\bot=\overline{ rg \big((\mathcal{F}^\alpha)^*\big)} \oplus rg \big({(\mathcal{F}^\alpha)^*}\big)^\bot=\overline{rg \big((\mathcal{F}^\alpha)^*\big)}\oplus ker (\mathcal{F}^\alpha).$$
For a scale vector satisfying the given condition, $\mathcal{F}^\alpha$ is injective. That is, $ker (\mathcal{F}^\alpha)=\{0\}$. Consequently,
$\mathcal{H}^k(I)=\overline{ rg \big((\mathcal{F}^\alpha)^*\big)}$ , i.e., $rg \big((\mathcal{F}^\alpha)^*\big)$ is dense in  $\mathcal{H}^k(I)$. For a linear and bounded operator of a Hilbert space, its range is closed if and
only if the range of its adjoint is closed (see, for instance, \cite{JC}). Therefore, $rg\big((\mathcal{F}^\alpha)^*\big)$ is closed and hence $rg \big((\mathcal{F}^\alpha)^*\big)= \mathcal{H}^k(I)$.
\end{proof}
\begin{theorem}
For a scale vector $\alpha$ satisfying $\Big[\sum_{i \in \mathbb{N}_{N-1}} \frac{|\alpha_i|^2}{a_i^{2k-1}}\Big]^{\frac{1}{2}}<\big(1+\|I_d-L\|\big)^{-1}$, the fractal operator $\mathcal{F}^\alpha: \mathcal{H}^k(I) \to \mathcal{H}^k(I)$ is Fredholm with index zero.
\end{theorem}
\begin{proof}
With the stated assumption on $\alpha$, Theorem \ref{THM6} asserts that the operator  $\mathcal{F}^\alpha$ is an isomorphism. Therefore, by the following theorem (see \cite{Ramm})
$\mathcal{F}^\alpha$ is Fredholm.\\
A linear bounded operator $A$ is Fredholm if and only if $A=B+F$, where $B$ is an isomorphism and $F$ has finite rank.\\ Also,
$$ \text{ind}~ \mathcal{F}^\alpha:= \text{dim}~ ker(\mathcal{F}^\alpha)- \text{codim}~rg(\mathcal{F}^\alpha)=0,$$ delivering the proof.
\end{proof}
\begin{remark}
Throughout the article we have confined our discussion to real-valued functions. However, we remark that many of our results apply immediately
to complex-valued functions defined on a real compact interval. For instance, denoting the real and imaginary parts of $f$ by $f_{re}$ and $f_{im}$ respectively,  one may consider the operator $$\mathcal{F}_C^\alpha: \mathcal{W}^{k,p}(I, \mathbb{C}) \to \mathcal{W}^{k,p}(I, \mathbb{C}), \quad \mathcal{F}_C^\alpha(f)=\mathcal{F}_C^\alpha(f_{re} + i f_{im}) = \mathcal{F}^\alpha (f_{re}) + i
\mathcal{F}^\alpha (f_{im}).$$
It is not hard to show that $\mathcal{F}_C^\alpha$ is a bounded linear operator. Other properties of $\mathcal{F}_C^\alpha$ can be similarly dealt with.
\end{remark}

\end{document}